\documentclass{commat}

\newcommand{\B}{{\cal B}}
\newcommand{\C}{{\mathbb C}}
\newcommand{\cA}{{\cal A}}
\newcommand{\cF}{{\cal F}}
\newcommand{\cG}{{\cal G}}

\title{Skew characteristic polynomial of graphs and embedded graphs}

\author{Riya Dogra and Sergei Lando}

\abstract{
	We introduce a new one-variable polynomial invariant of graphs, which we call the skew characteristic polynomial. For an oriented simple graph, this is just the characteristic polynomial of its anti-symmetric adjacency matrix. For non-oriented simple graphs the definition is different, but for a certain class of graphs (namely, for intersection graphs of chord diagrams), it gives the same answer if we endow such a graph with an orientation induced by the chord diagram.

    We prove that this invariant satisfies Vassiliev's $4$-term relations and determines therefore a finite type knot invariant. We investigate the behavior of the polynomial with respect to the Hopf algebra structure on the space of graphs and show that it takes a constant value on any primitive element in this Hopf algebra.

    We also provide a two-variable extension of the skew characteristic polynomial to embedded graphs and delta-matroids. The $4$-term relations for the extended polynomial prove that it determines a finite type invariant of multi-component links.
	}



\authorinfo[S. Lando]{Higher School of Economics, Skolkovo Institute of Science and Technology, Russia}{lando@hse.ru}
\authorinfo[R. Dogra]{Vrije Universiteit Amsterdam, Netherlands}{rdogra2613@gmail.com}

\VOLUME{31}
\NUMBER{3}
\YEAR{2023}
\firstpage{87}
\DOI{https://doi.org/10.46298/cm.11310}

\begin{document}

\begin{flushright}
To the memory of Sergei Duzhin\\
(1956--2015)
\end{flushright}

\section{Introduction}

To each graph, one can associate its adjacency matrix.
The characteristic polynomial of the adjacency matrix, as well as its roots,
are important invariants of graphs, and their study constitutes
the spectral theory of graphs.

The goal of the present paper is to introduce and start the study of another polynomial
graph invariant, which we call the \emph{skew characteristic polynomial}.
For an oriented graph, the skew characteristic polynomial
is just the characteristic polynomial of its antisymmetric adjacency matrix.
For a~non-oriented graph, however, the relationship between
its adjacency matrix and the skew characteristic polynomial is less
straightforward.

Most directly the skew characteristic polynomial can be defined for
intersection graphs of chord diagrams. Chord diagrams appeared in
V.~Vassiliev's theory of finite type invariants of knots as a~tool
for describing these invariants. Their intersection graphs admit
a family of natural orientations, and it happens that for all
orientations in this family the characteristic polynomial of
the antisymmetric adjacency matrix of the resulting oriented graph is the same.
We show that these characteristic polynomials satisfy Vassiliev's
$4$-term relations and determine thus a~knot invariant.

For an arbitrary intersection graph of a~chord diagram,
we define its skew characteristic polynomial
as this common characteristic polynomial.
We extend it by linearity to the Hopf algebra of chord diagrams modulo
$4$-term relations. It happens that for graphs with at least two vertices,
this polynomial becomes a~constant when restricted to the subspace of primitive elements in
the Hopf algebra.

The free term of the skew characteristic polynomial of an intersection
graph is either~$1$ or~$0$ depending on whether the adjacency matrix
of this graph is or is not non-degenerate over the field of two elements.
This function on graphs, called the \emph{non-degeneracy} of a~graph,
is naturally extended to arbitrary graphs, not necessarily intersection graphs.
Simple graphs also span a~Hopf algebra, and we define the skew characteristic
polynomial of graphs as the multiplicative polynomial graph invariant
whose value on primitives of degree at least~$2$ is a~constant,
and whose free term coincides with the non-degeneracy of the graph.
We show that this new graph invariant satisfies $4$-term relations for graphs.

Graphs embedded in two-dimensional surfaces do not span a~Hopf algebra,
which does not allow one to define their skew characteristic polynomial
directly. However, a~construction due to A.~Bouchet associates to
an embedded graph its delta-matroid, and delta-matroids span a~Hopf algebra.
Since non-degeneracy can be naturally defined for embedded graphs and depends
only on its delta-matroid, the skew characteristic polynomial admits a
natural extension to delta-matroids, whence to embedded graphs.
We show that the skew characteristic polynomial of embedded graphs satisfies
$4$-term relations for them and determines thus a~finite type link invariant.

Everywhere in the paper the ground field is~$\C$, the field of complex numbers,
if otherwise is not stated explicitly.

The paper is organized as follows. In Sec.~\ref{s1}, we
give the definition
of skew characteristic polynomial for chord diagrams and graphs.
In Sec.~\ref{s2}, we recall the definitions of $4$-term relations
for chord diagrams and graphs and show that the skew characteristic
polynomial satisfies these $4$-term relations determining thus a
finite type knot invariant.

Section~\ref{s3} is devoted to the definition of the skew characteristic polynomial
of embedded graphs and delta-matroids. We show that this polynomial satisfies $4$-term relations
and defines, therefore, a~finite type invariant of multicomponent links.

Section~\ref{s4} addresses a~number of natural questions and problems about
the invariant we introduce.

The authors are grateful to M.~Kazarian for valuable suggestions.
The work of S.~Lando was partly
funded by the RSF grant 23-11-00150 Mathematics of modern mathematical physics.

\section{Skew characteristic polynomial for graphs}\label{s1}

In this section we give the definition of the skew characteristic polynomial
for graphs. The definition is based on the notion of non-degeneracy of a~graph,
and the structure of Hopf algebra of graphs. We show that for intersection
graphs of chord diagrams the skew characteristic polynomial thus
defined coincides with the characteristic polynomial of the antisymmetric adjacency
matrix of the intersection graph supplied with a~natural orientation.

\subsection{Definition of the skew characteristic polynomial}
Given an abstract oriented simple graph $\vec G=(V,\vec E)$ with vertex set $V$,
and edge set $\vec E$, its adjacency matrix $A_{\vec G}$, is defined as
\[
    A_{\vec G} = (a_{ij}), \text{ where }a_{ij}=
\begin{cases}
            \phantom{-}1, & \text{if $(i,j) \in \vec E(G)$;}\\
            -1, & \text{if $(j,i) \in \vec E(G)$;}\\
             \phantom{-}0, & \text{otherwise.}
            \end{cases}
\]

The adjacency matrix of an oriented graph is antisymmetric.
We define the skew characteristic polynomial of an abstract oriented simple graph as the characteristic polynomial of its adjacency matrix,
that is

\begin{definition}
Let $\vec G$ be an abstract oriented graph and let $A_{\vec G}$ be its adjacency matrix. The \emph{skew characteristic polynomial}
of the graph $\vec G$, denoted $Q_{\vec G}$, is defined as
\[
Q_{\vec G}(u):= \det(uI - A_{\vec G}).
\]
\end{definition}

Our goal now is to define the skew characteristic polynomial of an abstract simple \textit{non-oriented} graph.
The first step consists in defining another invariant, the non-degeneracy of the graph.

For a~simple graph $G=(V,E)$, its adjacency matrix $A_G$ is, as usual,
the symmetric matrix over the field of two elements $\{0,1\}$ given by
\[
    A_G = (a_{ij}), \text{ where }a_{ij} =
\begin{cases}
            \text{1}, & \text{if $(i,j) \in E(G)$;}\\
            0, & \text{otherwise.}
            \end{cases}
\]

\begin{definition}
The \emph{non-degeneracy} $\nu(G)$ of a~graph~$G$ is equal to~$1$ provided the
matrix~$A_G$ is non-degenerate over the field of two elements, and it is~$0$ otherwise.
By definition, the non-degeneracy of the empty graph is~$1$.
\end{definition}

In particular, the non-degeneracy of any graph with an odd number of vertices is~$0$,
only a~graph with even number of vertices can be non-degenerate.
Note that we consider non-degeneracy as taking values in the field~$\C$ of complex numbers,
not in the field of two elements.

\begin{definition}
The \emph{skew characteristic polynomial} $Q:G\mapsto Q_G(u)$ is the graph invariant taking values in the
ring $\C[u]$ of polynomials in a~single variable~$u$, defined by $Q_G(u)=\nolinebreak\sum_k q_k(G)u^{|V(G)|-k}$, where
\[
q_k(G)=\sum_{\substack{U\subset V(G),\\|U|=k}}\nu(G(U)),
\]
and $G(U)$ is the subgraph of~$G$ induced by the subset~$U$ of the set~$V(G)$
of its vertices.
\end{definition}

\begin{theorem}\label{thMP}
The skew characteristic polynomial possesses the following properties:
\begin{itemize}
\item The degree of $Q_G$ is the number of vertices $|V(G)|$ in~$G$;
\item The polynomial $Q_G$ is even (that is, contains no monomials of odd degree) if the number of vertices $|V(G)|$ in~$G$ is even,
and it is odd (containing no monomials of even degree) otherwise;
\item All the nonzero coefficients in~$Q_G$ are positive integers;
\item The leading coefficient of~$Q_G$ is~$1$, and the coefficient of $u^{|V(G)|-2}$
is the number $|E(G)|$ of edges in~$G$;
\item The free term $Q_G(0)$ coincides with the non-degeneracy $\nu(G)$ of~$G$;
\item It is multiplicative, that is, if $G=G_1\sqcup G_2$ is a~disjoint union of two graphs $G_1$ and $G_2$, then $Q_{G}=Q_{G_1}\cdot Q_{G_2}$;
\item For the graph with a~single vertex, the polynomial is equal to~$u$.
\end{itemize}
\end{theorem}

All the properties follow immediately from the definition.

\subsection{Hopf algebra of graphs}

We recall the structure of the Hopf algebra of graphs introduced by Joni and Rota in~\cite{JR}.
Let $\cG$ be the infinite dimensional vector space over $\C$ spanned freely by all graphs.
Any graph invariant with values in a~vector space is extended to~$\cG$ by linearity;
below, we denote this extension by the same letter as the invariant itself.

The \textit{product} of two graphs is defined as their disjoint union, and extended
to a~\emph{multiplication}~$m$ of linear combinations of graphs by linearity.
This makes $\cG$ into a~graded commutative algebra, where the grading is induced by the number of vertices, and the unit is the empty graph,
\[
\cG = \cG_0 \oplus \cG_1 \oplus \cG_2 \oplus \hdots.
\]

The vector space $\cG$ is endowed with a~coalgebra structure by the \textit{comultiplication}
\[
\mu : \cG \longrightarrow \cG \otimes \cG
\]
defined as the \emph{coproduct} of a~graph as follows:
\[
\mu(G) = \sum_{V_1 \subset V(G)} G(V_1) \otimes G(V(G)\setminus V_1),
\]
where $G(V_1)$ is the induced subgraph of $G$, with the set of vertices $V_1$, and the sum is taken over all subsets of the vertex set. This comultiplication is extended to $\cG$ by linearity. The multiplication $m$, and the comultiplication $\mu$
both respect the grading:

\[
m : \cG_{n_1} \otimes \cG_{n_2} \longrightarrow \cG_{n_1 + n_2}
\]
\[
\mu : \cG_n \longrightarrow (\cG_0 \otimes \cG_n) \oplus (\cG_1 \otimes \cG_{n-1}) \oplus \hdots (\oplus\cG_n \otimes \cG_0).
\]

\begin{theorem}[\cite{JR}]
    The multiplication and the comultiplication defined above make the algebra of graphs into a~commutative cocommutative Hopf algebra.
\end{theorem}

According to the Milnor--Moore theorem, any graded cocommutative
Hopf algebra is a~polynomial algebra in its primitive elements.
An element~$p$ in a~Hopf algebra is said to be \emph{primitive} if
\[
\mu(p)=1\otimes p+p\otimes1.
\]
Primitive elements form a~graded vector subspace in a~graded Hopf algebra.
We denote the subspace of primitive elements in~$\cG_n$ by $P(\cG_n)$.

Monomials in elements of grading smaller than~$n$ span the subspace $D(\cG_n)\subset\cG_n$
of \emph{decomposable elements} in~$\cG_n$; in other words, decomposable elements are linear
combinations of disconnected graphs. Milnor--Moore theorem implies that
$\cG_n= P(\cG_n)\oplus D(\cG_n)$, for all $n=1,2,3,\dots$, and there is a~natural
projection $\pi_n:\cG_n\xrightarrow[]{\small{D(\cG_n)}}P(\cG_n)$ whose kernel is~$D(\cG_n)$.
Together, these projections~$\pi_n$ form the projection~$\pi:\cG\to P(\cG)$.

There is an explicit formula (\cite{lando1997primitive, S}) for this projection,
which expresses it as the logarithm $\pi=\log(\textrm{Id})$
of the identity mapping $\textrm{Id}:\cG\to\cG$: if~$G$ is a nonempty graph, then
\[
\pi:G\mapsto G-1!\sum_{I_1\sqcup I_2=V(G)} G(I_1)G(I_2)+2!\sum_{I_1\sqcup I_2\sqcup I_3=V(G)} G(I_1)G(I_2)G(I_3)-\dots,
\]
where each summation runs over the partitions of the set $V(G)$ of vertices of~$G$ into unordered nonempty parts.
The logarithm here is understood in the sense of the convolution product, that is, $\textrm{Id}$ is represented as the
sum $\textrm{Id}=1+\textrm{Id}_0$, where $1$ is the identity in $\cG_0$ and~$0$ in all $\cG_n$, $n\ge1$,
and $\textrm{Id}_0=\textrm{Id}-1$, so that
\[
\pi=\log(\textrm{Id})=\log(1+\textrm{Id}_0)=\textrm{Id}_0-\frac12(\textrm{Id}_0\otimes \textrm{Id}_0)\circ\mu+
\frac13(\textrm{Id}_0\otimes{\rm Id}_0\otimes{\rm Id}_0)\circ\mu^{\otimes2}-\dots.
\]

The following property describes the behavior of skew characteristic polynomial
with respect to the Hopf algebra structure.

\begin{theorem}
For any graph~$G$ with at least two vertices,
the value $Q_{\pi(G)}$ of the invariant~$Q$ on the projection $\pi(G)$ to the subspace of primitives
along the subspace of decomposable elements in the
Hopf algebra~$\cG$ of graphs is a~constant.
\end{theorem}

\begin{proof}
Let us make use of the formula~\cite{CKL2020, S}
\[
G=\sum_I\prod_{\iota\in I}\pi(G(\iota)),
\]
where the summation is carried over all partitions~$I=(\iota_1,\dots,\iota_k)$ of the set $V(G)$ of vertices
of the graph~$G$ into disjoint nonempty parts, and $G(\iota)$ is the subgraph of~$G$
induced by~$\iota\subset V(G)$. This formula allows one to reconstruct a~graph from the projection
to primitives of all its induced subgraphs. Separate on the right-hand side the projection
to primitives of~$G$ itself:
\begin{equation}\label{e1}
 G=\pi(G)+\sum_{I,|I|>1} \prod_{\iota\in I}\pi(G(\iota)).
\end{equation}

Now suppose the assertion of the theorem is true for all graphs with~$2,\dots,n$ vertices,
and let~$G$ be a~graph with~$n+1$ vertices. We want to prove that the coefficient of~$u^k$
in~$Q_{\pi(G)}$ is~$0$ for all~$k=1,\dots,n+1$.

The coefficient of $u^k$ in $Q_G$ is $\sum_{U,|U|=n+1-k}\nu(G(U))$.
Pick a~subset $U\in V(G)$ such that $|U|=n+1-k$. Its complement $V(G)\setminus U$ consists of~$k$ vertices.

By the induction hypothesis, a~partition~$\iota$ of the set $V(G)$ of vertices of~$G$
in the second summand in the right-hand side of equation~\eqref{e1} can contribute to
the coefficient of~$u^k$ only if this partition contains exactly~$k$ sets of size~$1$.
For a~given choice of $1$-element parts, let~$U$ denote the
complimentary set of vertices. Then the contribution to the coefficient of~$u^k$ is as follows:
compute the free term of the complementary sum
\[
 \sum_{J} \prod_{j\in J}\pi(G(U)(j)),
\]
where the summation runs over all partitions of the set~$U$ into nonempty disjoint subsets.

Equation~\eqref{e1} asserts that the last expression is nothing but the graph~$G(U)$,
and the free term of the polynomial $Q_{G(U)}$, which is equal to $\nu(G(U))$,
coincides with the free term of the skew characteristic polynomial of the sum on the right.

Hence, for each~$k=1,\dots,n+1$, the contribution of the non-degeneracy of an induced subgraph
with~$n+1-k$ vertices to the coefficient of~$u^k$ in the skew characteristic polynomial $Q_G$
coincides with the contribution of its complement to all the terms on the right in equation~\eqref{e1}
but the first one. This means that the contribution to~$u^k$ of $Q_{\pi(G)}$ is~$0$, as required.
\end{proof}

\begin{remark}
The unknown referee attracted our attention to the fact that the assertion of the theorem
follows from the representation of the skew-characteristic polynomial as the convolution
of non-degeneracy and the $4$-invariant taking a~graph~$G$ to $u^{|V(G)|}$.
\end{remark}

\subsection{Justification: skew characteristic polynomial for intersection graphs of chord diagrams}

The goal of this section is to justify the definition of the skew characteristic polynomial
of graphs given in the previous one. We define here the skew characteristic polynomial
of chord diagrams and show that it depends only on the intersection graph of the chord diagram
and coincides with the skew characteristic polynomial of this graph.
Chord diagrams carry considerable combinatorial information about Vassiliev, or finite type, invariants of knots,
whence their study often becomes imperative in learning about knots. We recall some preliminaries.
For a~detailed study, see, for example,~\cite{Chmutov2012IntroductionTV} or~\cite{Lando2003GraphsOS},

\begin{definition}
A \emph{chord diagram} (of order $n$) is an oriented circle together with a~set of $n$ disjoint pairs of distinct points considered up to diffeomorphisms of the circle that preserve the orientation.
\end{definition}

In a~graphical presentation of chord diagrams, the two points in each pair are usually
connected by a~chord (which is shown as either a~segment or an arc).

\begin{definition}[Intersection graph of a~chord diagram]
Let $C$ be a~chord diagram. The \emph{intersection graph} $g(C)$ has the vertex set
whose elements are in one-to-one correspondence
with the chords of $C$; there is an edge between two vertices if and only if the corresponding chords
intersect one another in $C$,
i.e., their ends lie on the circle in the alternating order.
\end{definition}

The intersection graph $g(C)$ is a~simple non-oriented graph.
Take a~point on the underlying circle of~$C$ different from all the ends
of the chords; we call this point `a cut point'.
 A choice of a~cut point allows one to orient $g(C)$ as follows.

We cut the chord diagram $C$ at the cut point and get a~linear chord diagram
(also called `long chord diagram' from the terminology of long knots, or `arc diagram')
as follows: the circle turns into an oriented line, and the chords become arcs with ends on this line.
Then we number the chords from $1$ to~$n$ according to the order in which
the left ends of the corresponding arcs follow the underlying line
and orient each edge in~$g(C)$ from the vertex with the smaller number to that with the greater one.
The resulting oriented graph will be denoted by $\overrightarrow{g(C)}$.
Of course, the orientation of the intersection graph depends on the choice of the cut point.
Nevertheless, the characteristic polynomial of the antisymmetric adjacency matrix
of $\overrightarrow{g(C)}$ is independent of this choice.

\begin{proposition}\label{pcpi}
The skew characteristic polynomial
\[
\det\left(uI-A_{\overrightarrow{g(C)}}\right)
\]
of the oriented characteristic graph of the chord diagram~$C$ does not
depend on the choice of the cut point determining the orientation.
\end{proposition}

\begin{proof}
In order to prove Proposition~\ref{pcpi}, it suffices to show that the characteristic
polynomial of the adjacency matrix of the oriented intersection graph $\overrightarrow{g(C)}$
remains unchanged under moving the cut point to the neighboring arc of the circle
of~$C$ in the positive direction. Such a~move results in a~renumbering of the chords;
all the oriented edges incident to the former vertex~$1$ change their orientation.
The adjacency matrix $A_{\overrightarrow{g(C)}}$
of the oriented intersection graph is replaced with the matrix
(up to a~simultaneous cyclic renumbering of the columns and the rows of the matrix)
\[
I_1\cdot A_{\overrightarrow{g(C)}}\cdot I_1,
\]
where~$I_1$ is the identity $n\times n$ matrix with~$1$ replaced by~$-1$ in the entry $(1,1)$.
Of course, such a~move does not change the characteristic polynomial of the matrix,
since $I_1=I_1^{-1}$.
\end{proof}

This property allows us to give the following definition.

\begin{definition}[Skew-characteristic polynomial of a~chord diagram]
Let~$C$ be a~chord diagram of order $n$. The \emph{skew characteristic polynomial} $Q_C(u)$ of the chord diagram $C$
is the skew characteristic polynomial $Q_{\overrightarrow{g(C)}}(u)$ of the oriented intersection graph of~$C$,
which is independent of the cut point determining the orientation.
\end{definition}

\begin{theorem}\label{thFT}
The free term $Q_C(0)$ of the characteristic polynomial of a~chord diagram~$C$ coincides
with the non-degeneracy $\nu(g(C))$ of the intersection graph $g(C)$ of~$C$.
\end{theorem}

\begin{proof}
This statement is, essentially, in~\cite[Theorem 3]{BarNatan1996OnTM}, which states that
the determinant $\det~A_{\overrightarrow{g(C)}}$ is either~$1$ or~$0$ depending on whether
the thickening of the chords of~$C$ provides a~connected or a~disconnected curve, respectively.
This invariant is called the Conway weight system in~\cite{BarNatan1996OnTM},
since it arises from the Conway polynomial of knots.
The fact that the resulting curve is connected if and only if the adjacency matrix $A_{g(C)}$ is
non-degenerate over the field of two elements is proved, for example, in~\cite{S01}.
\end{proof}

\begin{corollary}
The skew characteristic polynomial of a~chord diagram~$C$
coincides with the skew characteristic polynomial $Q_{g(C)}$ of its intersection graph $g(C)$.
In particular, it depends on the intersection
graph~$g(C)$ of~$C$ rather than on the chord diagram~$C$ itself; that is, if two chord
diagrams~$C_1$ and $C_2$ have isomorphic intersection graphs, then their skew characteristic
polynomials coincide, i.e.~$Q_{C_1}=Q_{C_2}$.
\end{corollary}

\begin{proof}
The coefficient of $u^{n-k}$ in the characteristic polynomial of the anti-symmetric adjacency matrix
$A_{\overrightarrow{g(C)}}$ is the sum over all $k$-chord subdiagrams of~$C$ of the non-degeneracies
of these subdiagrams. Since the non-degeneracy of a~chord diagram coincides with the non-degeneracy
of its intersection graph, this coefficient coincides with $q_k(g(C))$.
\end{proof}

For a~chord diagram~$C$, the value of the skew characteristic polynomial on its projection~$\pi(C)$
to the subspace of primitives admits a~nice combinatorial interpretation.
This value is an integer, which is~$0$ if the number of chords in~$C$ is odd.
If the number of chords is even, then a~sign can be assigned to each
Hamiltonian cycle in $\overrightarrow{g(C)}$: the sign is~$+1$ provided
the number of arrows in this cycle pointing in the same direction is even, and it is~$-1$
otherwise. For the proof of the following statement, which is an assertion about
the non-degeneracy of chord diagrams, see~\cite{BarNatan2015ProofOA}.

\begin{proposition}
The value $Q_{\pi(C)}$ of the skew characteristic polynomial on the
projection of a~chord diagram~$C$ with an even number of chords
to the subspace of primitives is twice the difference between the number of
positive Hamiltonian cycles and the number of negative Hamiltonian
cycles in $\overrightarrow{g(C)}$, which is independent of the cut point in~$C$
determining the orientation of~$g(C)$.
\end{proposition}

Note that \emph{no orientation of the $5$-wheel or $3$-prism} (see Fig.~\ref{fR3}) \emph{leads to an
oriented graph for which twice the number of Hamiltonian cycles counted with signs coincides
with the value of the non-degeneracy on the projection of the graph to the subspace of primitives}.
To prove this for the $5$-wheel, for example, we remark that each of its edges
belongs to an even number of Hamiltonian cycles. This means that changing
the orientation of an edge preserves residue of the number of Hamiltonian cycles
modulo~$4$. Now, picking an arbitrary orientation of the edges in the $5$-wheel,
we can compute this residue, which is~$-1$. Hence, under neither choice of the
orientation this number can be~$-3$, while the value of~$Q$ on the
projection of the $5$-wheel to primitives is~$-6$.
This means, in particular, that the skew characteristic polynomial of
non-oriented graphs cannot be reduced to the characteristic polynomial
of these graphs with a~certain orientation chosen.

In~\cite{BarNatan2015ProofOA, Kulakova2014OnAW}, the number of Hamiltonian cycles
in oriented intersection graphs of chord diagrams counted
with signs is related to the
weight system defined by the Lie algebra $sl(2)$, which comes from the
well-known knot invariant called the
colored Jones polynomial.

\subsection{Examples}

\begin{proposition}
For complete graphs~$K_n$ on~$n$ vertices, we have
\[
Q_{K_n}(u)=\sum_{k=0}^{\lfloor n/2 \rfloor} \genfrac{(}{)}{0pt}{}{n}{2k}u^{n-2k}.
\]
\end{proposition}

Indeed, the non-degeneracy of a~complete graph with odd number of vertices is~$0$,
and it is~$1$ if the number of vertices is even. A graph on~$n$ vertices
contains $\genfrac{(}{)}{0pt}{}{n}{2k}$ induced subgraphs with~$2k$ vertices, and for a~complete graph~$K_n$
all the induced subgraphs also are complete graphs.

The sequence of polynomials $Q_{K_n}$, for $n=0,1,2,\dots$ starts with
\[
1,\quad u,\quad u^2 +1,\quad u^3 +3u,\quad u^4 +6u^2 +1,\quad u^5 +10u^3 +5u,\quad u^6 +15u^4 +15u^2 +1,\dots
\]

Note that a~complete graph~$K_n$ is the intersection graph of a~unique chord diagram,
the one with~$n$ chords, and any choice of the cut point leads to the same orientation, up to isomorphism:
for an arbitrary numbering of the vertices each edge is oriented from a~vertex with a
smaller number to that with the greater one.

\begin{proposition}
For complete bipartite graphs~$K_{m,n}$, we have
\[
Q_{K_{m,n}}(u)=u^{m+n}+mnu^{m+n-2}.
\]
\end{proposition}

Indeed, each induced subgraph of a~complete bipartite graph is a~complete
bipartite graph, and the non-degeneracy of a~complete bipartite graph
having at least~$2$ elements in one of its parts is~$0$.

Similarly to complete graphs, for a~complete bipartite graph~$K_{m,n}$ there is unique
chord diagram having it as the intersection graph: a~family of $m$ mutually parallel chords
intersects transversally another family of~$n$ mutually parallel chords. However, in contrast to
complete graphs, the intersection graph $K_{m,n}$ admits several non-isomorphic orientations
depending on the choice of the cut point. If we choose the cut point between
the two families of mutually parallel chords constituting the
two parts of the graph, then all the edges will be oriented from
one part to the other one. This orientation allows one to easily compute
$Q_{K_{m,n}}(u)$.

Now let us compute the skew characteristic polynomial of the two graphs with~$6$
vertices that are not intersection graphs, namely, the $5$-wheel and the $3$-prism. These graphs are shown in Figure~\ref{fR3}.
They are the smallest graphs possessing this property.

\begin{figure}[h]
\begin{center}
\begin{picture}
(300,60)(10,40)

\thicklines

\put(55,51){\circle*{4}}
\put(55,75){\circle*{4}}
\put(34,62){\circle*{4}}
\put(76,62){\circle*{4}}
\put(40,35){\circle*{4}}
\put(70,35){\circle*{4}}
\put(55,50){\line(0,1){24}}
\put(55,51){\line(2,1){21}}
\put(55,51){\line(-2,1){21}}
\put(55,50){\line(1,-1){16}}
\put(55,50){\line(-1,-1){16}}
\put(40,35){\line(1,0){32}}
\put(55,75){\line(2,-1){21}}
\put(55,75){\line(-2,-1){21}}
\put(40,35){\line(-1,4){7}}
\put(70,35){\line(1,4){7}}

\put(255,70){\circle*{4}}
\put(255,40){\circle*{4}}
\put(235,80){\circle*{4}}
\put(275,80){\circle*{4}}
\put(235,30){\circle*{4}}
\put(275,30){\circle*{4}}
\put(235,30){\line(0,1){50}}
\put(235,30){\line(1,0){40}}
\put(235,80){\line(1,0){40}}
\put(275,30){\line(0,1){50}}
\put(255,40){\line(0,1){30}}
\put(235,30){\line(2,1){20}}
\put(235,80){\line(2,-1){20}}
\put(275,80){\line(-2,-1){20}}
\put(275,30){\line(-2,1){20}}

\end{picture}
\end{center}

		\caption{The $5$-wheel and the $3$-prism}
		\label{fR3}
\end{figure}

\begin{proposition}
The skew characteristic polynomial of the $5$-wheel is
\[
u^6 +10u^4 +10u^2 +1,
\]
while that of the $3$-prism is
\[
u^6 +9u^4 +12u^2.
\]
\end{proposition}

These results can be obtained directly from the definition.

\subsection{Refined skew characteristic polynomial}

The number of connected components of the boundary of a~thickened chord diagram
is~$1$ if and only if the intersection graph of the chord diagram is non-degenerate.
This statement can be generalized:

\textit{The number of connected components of the boundary of a~thickened chord diagram~$C$
is one greater than the corank of the adjacency matrix $A_{g(C)}$ of its intersection graph
$g(C)$, considered as a~matrix over the field of two elements} (see, e.g.~\cite{S01}).

Following the suggestion of M.~Kazarian, we define the \emph{refined skew characteristic
polynomial} $\overline{Q}_G(u,v)$ in two variables $u,v$ of a~graph~$Q$ by the formula
\[
\overline{Q}_G(u,v)=\sum_{U\subset V(G)}u^{|V(G)|-|U|}v^{\textrm{corank}A_{G(U)}}.
\]
The $v$-free term of the refined skew characteristic polynomial coincides with
the original skew characteristic polynomial, $\overline Q_G(u,0)=Q_G(u)$, for any graph~$G$.

\begin{proposition}
For complete graphs~$K_n$, we have
\[
\overline Q_{K_n}(u,v)=\sum_{k=0}^{\lfloor n/2 \rfloor}\genfrac{(}{)}{0pt}{}{n}{2k}u^{n-2k}+v\sum_{k=0}^{\lfloor (n-1)/2 \rfloor}\genfrac{(}{)}{0pt}{}{n}{2k+1}u^{n-2k-1}.
\]
\end{proposition}

Indeed, the corank of the adjacency matrix of a~complete graph on odd number of vertices is~$1$.

For a~chord diagram, one can define its refined skew characteristic polynomial either
as the refined skew characteristic polynomial of its intersection graph, or, equivalently, as a~sum
of the corresponding monomials over all chord subdiagrams, with the degree of the variable~$w$
equal to the number of connected components of the boundary of the thickened subdiagram,
minus~$1$.

\section{\texorpdfstring{$4$}{4}-term relations and knot invariants}\label{s2}

One of the fundamental results in the theory of invariants of finite type, due to M.~Kontsevich, is that
over a~field of characteristic~$0$ every `weight system' of order $n$ comes from a~Vassiliev invariant of type $n$,
which is well-defined up to Vassiliev invariants of order at most $n-1$, see~\cite{BarNatan1995OnTV}.
The goal of this section is to show that the skew characteristic polynomial for chord
diagrams, as well as that for graphs, satisfies the $4$-term relation and defines thus a~knot invariant.
The same is true for the refined skew characteristic polynomial.

\subsection{4-term relations for chord diagrams and matrices}
\begin{definition}
A function~$f$ on chord diagrams is a~\emph{weight system} if it satisfies the 4-term relation,
i.e. the following alternating sum is equal to zero.
\end{definition}

\begin{center}
\begin{tabular}
{c c c c c c c c}
    $f$ \Big(\raisebox{-.4\height}{\includegraphics[width=1.1cm, height=1.45 cm]{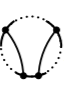}}\Big) & - & $f$\Big(\raisebox{-.4\height}{\includegraphics[width=1.2cm, height=1.25cm]{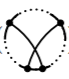}}\Big) & + & $f$\Big(\raisebox{-.4\height}{\includegraphics[width=1.2cm, height=1.35cm]{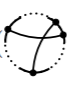}}\Big) & - & $f$\Big(\raisebox{-.4\height}{\includegraphics[width=1.1cm, height=1.40cm]{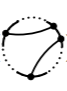}}\Big) & $=0$
\end{tabular}
\end{center}

In the above diagrams, the chords whose end points lie on the dotted part of the circle remain fixed throughout the expression.
Replacement of the first chord diagram in the $4$-term relation with the second one is called the \emph{first Vassiliev move},
and with the fourth one the \emph{second Vassiliev move}. Such a~move is defined by a~pair of chords having two neighboring ends.
We denote by~$\cA_n$ the vector space spanned by all chord diagrams with~$n$ chords modulo the subspace spanned by linear
combinations of chord diagrams entering the $4$-term relations.

We now want to give the vector space $\cA=\oplus_{n\geq 0}\cA_n$, the structure of an algebra.
Let $C_1$, $C_2$ be two chord diagrams of orders $m$ and $n$ respectively. Define their \textit{product} as their connected sum, i.e., for each diagram, mark two points on the circle such that the marked arc does not have an endpoint of a~chord; then join $C_1$ and $C_2$ along those arcs respecting the orientation (in all the figures,
the outer circle of the diagrams is assumed to be oriented counter-clockwise by default) as shown in the example below:

\begin{center}
   \raisebox{-.5\height}{\includegraphics[width=2.8cm, height=1.25cm]{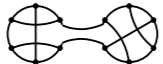}} $=$ \raisebox{-.5\height}{\includegraphics[width=1.5cm, height=1.5cm]{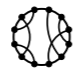}}
\end{center}

Extend this map using linearity to get the mapping
\[
m:\cA_m\otimes\cA_n\rightarrow\cA_{m+n}.
\]

The following statement is well-known, see, e.g.~\cite{Lando2003GraphsOS}.

\begin{lemma}
The product $m$ is well-defined modulo the 4-term relations.
\end{lemma}

\textit{Comultiplication} of chord diagrams is defined in terms of the \emph{coproduct} in $\cA$ as follows
\[
\mu:\cA_n\rightarrow\bigoplus_{k+l=n}\cA_k\otimes\cA_l
\]
\[
C \mapsto\sum_{J\subset V(C)}C_J\otimes C_{\overline{J}},
\]
where $V(C)$ is the set of chords of $C$ and $\overline{J}=V(C)\setminus J$ and $\mu$ is extended to $\cA$ by linearity.

Note that in contrast to multiplication, comultiplication does not require $4$-term relations to be well-defined:
it introduces a~coalgebra structure on the space of chord diagrams even if we do not factor out the $4$-term relations.

\begin{theorem}
The skew characteristic polynomials of chord diagrams satisfy the $4$-term relations and
whence determine a~knot invariant.
\end{theorem}

This theorem can be either proved directly (basing on the fact that the non-degeneracy of the intersection
graph of a~chord diagram satisfies the $4$-term relation) or deduced from a~similar theorem
for skew characteristic polynomial of graphs, see below. One more approach to proving this
theorem consists in proving its analogue for antisymmetric matrices, see Sec.~\ref{ssasm}.

\subsection{4-term relation for graphs}\label{ss4t}

The $4$-term relation for graphs is defined in a~way similar to that for chord diagrams, see~\cite{Lando2000OnAH}.
A \textit{graph invariant} is a~function on isomorphism classes of graphs attaining values in a~commutative group or a~commutative ring. Let $G=(V,E)$ be a~simple graph, with the vertex set $V=V(G)$ and the edge set $E=E(G)$.
Let $a,b \in V(G)$ be distinct. Define $G_{ab}'$ to be the graph obtained by deleting the edge $(a,b) \in E(G)$
if it exists and adding the edge otherwise (switching the adjacency);
this is the first Vassiliev move for graphs. The graph
$\widetilde{G}_{ab}$ is obtained in the following way.
 Let $c\in V(G)\backslash\{a,b\}$. If $c$ is connected to $b$, then we change its adjacency with $a$;
 we do nothing otherwise. This is the second Vassiliev move.
 It is straightforward to see that the operations $G\mapsto G_{ab}'$ and $G\mapsto\widetilde{G}_{ab}$ commute.

\begin{definition}
    A graph invariant is a~4-\emph{invariant} if it satisfies the \emph{four-term relation}
    \[
f(G)-f(G_{ab}')=f(\widetilde{G}_{ab})-f(\widetilde{G}_{ab}')
\]
for an arbitrary graph $G$, and any pair of vertices $a,b\in V(G)$.
\end{definition}
The space $\cG$ modulo the $4$-term relations
    \[
G - G_{ab}' - \widetilde{G}_{ab} + \widetilde{G}_{ab}' = 0
\]
is denoted by $\cF$. It inherits the Hopf algebra structure from $\cG$.
The grading on $\cF$ is induced by the number of vertices since the $4$-term relation consists of terms with the same number of vertices:
    \[
\cF = \cF_0 \oplus \cF_1 \oplus \cF_2 \oplus \hdots,
\]
where $\cF_k$ is the subspace spanned by graphs with $k$ vertices, modulo the $4$-term relations.
The mapping taking a~chord diagram to its intersection graph preserves the $4$-term relations and
is thus extended by linearity to
a graded Hopf algebra isomorphism $\cA\to\cF$.
As a~consequence, any $4$-invariant of graphs determines a~weight system and hence a~finite type invariant of knots.

The non-degeneracy of graphs is a~$4$-invariant; moreover, it satisfies the $2$-term relation:
\[
\nu(G)=\nu(\widetilde{G}_{ab})
\]
for any simple graph~$G$ and any pair of vertices $a,b\in V(G)$.
Note that the proof of Theorem~\ref{thFT} in~\cite{BarNatan1996OnTM}
shows that the determinant of the adjacency matrix also satisfies the $2$-term relation for
chord diagrams.
This fact immediately implies

\begin{theorem}
The skew characteristic polynomial of simple graphs is a~$4$-invariant.
\end{theorem}

Knowing that the skew characteristic polynomial of graphs satisfies $4$-term relations
produces one more way to compute it for those graphs that are not intersection graphs.
If such a~graph can be expressed, modulo $4$-term relations, as a~linear combination of intersection graphs,
then the same linear combination of their skew characteristic polynomials
yields its skew characteristic polynomial.
For the $5$-wheel and the $3$-prism, this is shown in Figure~\ref{fR4} and Figure~\ref{fR5}.
However, any graph is known to be equivalent modulo $4$-term relations
to a~linear combination of intersection graphs only for graphs with up to~$8$ vertices, see~\cite{K21}.
For an arbitrary number of vertices, the question remains open.

\begin{figure}[ht]
\begin{center}
\setlength{\unitlength}{0.8pt}
\begin{picture}
(300,250)(10,120)\small

\thicklines

\put(55,351){\circle*{4}}
\put(55,375){\circle*{4}}
\put(34,362){\circle*{4}}
\put(76,362){\circle*{4}}

\put(40,335){\circle*{4}}
\put(70,335){\circle*{4}}
\put(55,350){\line(0,1){24}}
\put(55,351){\line(2,1){21}}

\put(55,351){\line(-2,1){21}}
\put(55,350){\line(1,-1){16}}
\put(55,350){\line(-1,-1){16}}
\put(40,335){\line(1,0){32}}
\put(55,375){\line(2,-1){21}}
\put(55,375){\line(-2,-1){21}}
\put(40,335){\line(-1,4){7}}
\put(70,335){\line(1,4){7}}

\put(90,351){$-$}
\put(28,332){$A$}
\put(72,332){$B$}

\put(135,351){\circle*{4}}
\put(135,375){\circle*{4}}
\put(114,362){\circle*{4}}
\put(156,362){\circle*{4}}
\put(120,335){\circle*{4}}
\put(150,335){\circle*{4}}
\put(135,350){\line(0,1){24}}
\put(135,351){\line(2,1){21}}
\put(135,351){\line(-2,1){21}}
\put(135,350){\line(1,-1){16}}
\put(135,350){\line(-1,-1){16}}
\put(135,375){\line(2,-1){21}}
\put(135,375){\line(-2,-1){21}}
\put(120,335){\line(-1,4){7}}
\put(150,335){\line(1,4){7}}

\put(170,351){$=$}
\put(108,332){$A$}
\put(152,332){$B$}

\put(215,351){\circle*{4}}
\put(215,375){\circle*{4}}
\put(194,362){\circle*{4}}
\put(236,362){\circle*{4}}
\put(200,335){\circle*{4}}
\put(230,335){\circle*{4}}
\put(215,350){\line(0,1){24}}
\put(215,351){\line(2,1){21}}
\put(215,351){\line(-2,1){21}}
\put(215,350){\line(1,-1){16}}
\put(200,336){\line(3,2){36}}
\put(200,335){\line(1,0){32}}
\put(215,375){\line(2,-1){21}}
\put(215,375){\line(-2,-1){21}}
\put(200,335){\line(-1,4){7}}
\put(230,335){\line(1,4){7}}

\put(250,351){$-$}
\put(188,332){$A$}
\put(232,332){$B$}

\put(295,351){\circle*{4}}
\put(295,375){\circle*{4}}
\put(274,362){\circle*{4}}
\put(316,362){\circle*{4}}
\put(280,335){\circle*{4}}
\put(310,335){\circle*{4}}
\put(295,350){\line(0,1){24}}
\put(295,351){\line(2,1){21}}
\put(295,351){\line(-2,1){21}}
\put(295,350){\line(1,-1){16}}
\put(280,336){\line(3,2){36}}
\put(295,375){\line(2,-1){21}}
\put(295,375){\line(-2,-1){21}}
\put(280,335){\line(-1,4){7}}
\put(310,335){\line(1,4){7}}

\put(268,332){$A$}
\put(312,332){$B$}

\put(100,315){${\scriptstyle u^6 +9u^4 +9u^2 +1}$}
\put(180,315){${\scriptstyle u^6 +10u^4 +9u^2 +1}$}
\put(270,315){${\scriptstyle u^6 +9u^4 +8u^2 +1}$}

\put(55,251){\circle*{4}}
\put(55,275){\circle*{4}}
\put(34,262){\circle*{4}}
\put(76,262){\circle*{4}}
\put(40,235){\circle*{4}}
\put(70,235){\circle*{4}}
\put(55,250){\line(0,1){24}}
\put(55,251){\line(2,1){21}}
\put(55,251){\line(-2,1){21}}
\put(55,250){\line(1,-1){16}}
\put(55,250){\line(-1,-1){16}}
\put(40,235){\line(1,0){32}}
\put(55,275){\line(2,-1){21}}
\put(55,275){\line(-2,-1){21}}
\put(40,235){\line(-1,4){7}}
\put(70,235){\line(1,4){7}}

\put(90,251){$-$}
\put(28,232){$A$}
\put(55,257){$B$}

\put(135,251){\circle*{4}}
\put(135,275){\circle*{4}}
\put(114,262){\circle*{4}}
\put(156,262){\circle*{4}}
\put(120,235){\circle*{4}}
\put(150,235){\circle*{4}}
\put(135,250){\line(0,1){24}}
\put(135,251){\line(2,1){21}}
\put(135,251){\line(-2,1){21}}
\put(135,250){\line(1,-1){16}}
\put(120,235){\line(1,0){32}}
\put(135,275){\line(2,-1){21}}
\put(135,275){\line(-2,-1){21}}
\put(120,235){\line(-1,4){7}}
\put(150,235){\line(1,4){7}}

\put(170,251){$=$}
\put(108,232){$A$}
\put(135,257){$B$}

\put(215,251){\circle*{4}}
\put(215,275){\circle*{4}}
\put(194,262){\circle*{4}}
\put(236,262){\circle*{4}}
\put(200,235){\circle*{4}}
\put(230,235){\circle*{4}}
\put(215,250){\line(0,1){24}}
\put(215,251){\line(2,1){21}}
\put(215,251){\line(-2,1){21}}
\put(215,250){\line(1,-1){16}}
\put(215,250){\line(-1,-1){16}}
\put(201,236){\line(1,3){13}}
\put(200,236){\line(3,2){36}}
\put(215,275){\line(2,-1){21}}
\put(215,275){\line(-2,-1){21}}
\put(230,235){\line(1,4){7}}

\put(250,251){$-$}
\put(188,232){$A$}
\put(215,257){$B$}

\put(295,251){\circle*{4}}
\put(295,275){\circle*{4}}
\put(274,262){\circle*{4}}
\put(316,262){\circle*{4}}
\put(280,235){\circle*{4}}
\put(310,235){\circle*{4}}
\put(295,250){\line(0,1){24}}
\put(295,251){\line(2,1){21}}
\put(295,251){\line(-2,1){21}}
\put(295,250){\line(1,-1){16}}
\put(280,236){\line(3,2){36}}
\put(281,236){\line(1,3){13}}
\put(295,275){\line(2,-1){21}}
\put(295,275){\line(-2,-1){21}}
\put(310,235){\line(1,4){7}}

\put(268,232){$A$}
\put(295,257){$B$}

\put(100,215){${\scriptstyle u^6 +9u^4 +11u^2}$}
\put(180,215){${\scriptstyle u^6 +10u^4 +8u^2 +1}$}
\put(270,215){${\scriptstyle u^6 +9u^4 +9u^2}$}

\put(55,151){\circle*{4}}
\put(55,175){\circle*{4}}
\put(34,162){\circle*{4}}
\put(76,162){\circle*{4}}
\put(40,135){\circle*{4}}
\put(70,135){\circle*{4}}
\put(55,150){\line(0,1){24}}
\put(55,151){\line(2,1){21}}
\put(55,151){\line(-2,1){21}}
\put(55,150){\line(1,-1){16}}
\put(55,150){\line(-1,-1){16}}
\put(40,135){\line(1,0){32}}
\put(55,175){\line(2,-1){21}}
\put(55,175){\line(-2,-1){21}}
\put(40,135){\line(-1,4){7}}
\put(70,135){\line(1,4){7}}

\put(90,151){$-$}
\put(25,132){$B$}
\put(55,157){$A$}

\put(135,151){\circle*{4}}
\put(135,175){\circle*{4}}
\put(114,162){\circle*{4}}
\put(156,162){\circle*{4}}
\put(120,135){\circle*{4}}
\put(150,135){\circle*{4}}
\put(135,150){\line(0,1){24}}
\put(135,151){\line(2,1){21}}
\put(135,151){\line(-2,1){21}}
\put(135,150){\line(1,-1){16}}
\put(120,135){\line(1,0){32}}
\put(135,175){\line(2,-1){21}}
\put(135,175){\line(-2,-1){21}}
\put(120,135){\line(-1,4){7}}
\put(150,135){\line(1,4){7}}

\put(170,151){$=$}
\put(105,132){$B$}
\put(135,157){$A$}

\put(215,151){\circle*{4}}
\put(215,175){\circle*{4}}
\put(194,162){\circle*{4}}
\put(236,162){\circle*{4}}
\put(200,135){\circle*{4}}
\put(230,135){\circle*{4}}
\put(215,150){\line(0,1){24}}
\put(215,151){\line(2,1){21}}
\put(215,150){\line(-1,-1){16}}
\put(200,135){\line(1,0){32}}
\put(215,175){\line(2,-1){21}}
\put(215,175){\line(-2,-1){21}}
\put(200,135){\line(-1,4){7}}
\put(230,135){\line(1,4){7}}

\put(250,151){$-$}
\put(185,132){$B$}
\put(215,157){$A$}

\put(295,151){\circle*{4}}
\put(295,175){\circle*{4}}
\put(274,162){\circle*{4}}
\put(316,162){\circle*{4}}
\put(280,135){\circle*{4}}
\put(310,135){\circle*{4}}
\put(295,150){\line(0,1){24}}
\put(295,151){\line(2,1){21}}
\put(280,135){\line(1,0){32}}
\put(295,175){\line(2,-1){21}}
\put(295,175){\line(-2,-1){21}}
\put(280,135){\line(-1,4){7}}
\put(310,135){\line(1,4){7}}

\put(265,132){$B$}
\put(295,157){$A$}

\put(100,115){${\scriptstyle u^6 +9u^4 +11u^2}$}
\put(180,115){${\scriptstyle u^6 +8u^4 +10u^2 +1}$}
\put(270,115){${\scriptstyle u^6 +7u^4 +11u^2}$}
\end{picture}
\end{center}
\caption{$4$-term relations for the $5$-wheel expressed in terms
of intersection graphs; the value of the skew characteristic polynomial is indicated} \label{fR4}
\end{figure}

\begin{figure}[ht!]
\begin{center}
\setlength{\unitlength}{0.8pt}
\begin{picture}
(300,180)(10,120)\small

\thicklines

\put(55,270){\circle*{4}}
\put(55,240){\circle*{4}}
\put(35,280){\circle*{4}}
\put(75,280){\circle*{4}}
\put(35,230){\circle*{4}}
\put(75,230){\circle*{4}}
\put(35,230){\line(0,1){50}}
\put(35,230){\line(1,0){40}}
\put(35,280){\line(1,0){40}}
\put(75,230){\line(0,1){50}}
\put(55,240){\line(0,1){30}}
\put(35,230){\line(2,1){20}}
\put(35,280){\line(2,-1){20}}
\put(75,280){\line(-2,-1){20}}
\put(75,230){\line(-2,1){20}}

\put(90,251){$-$}
\put(43,240){$A$}
\put(58,260){$B$}

\put(135,270){\circle*{4}}
\put(135,240){\circle*{4}}
\put(115,280){\circle*{4}}
\put(155,280){\circle*{4}}
\put(115,230){\circle*{4}}
\put(155,230){\circle*{4}}
\put(115,230){\line(0,1){50}}
\put(115,230){\line(1,0){40}}
\put(115,280){\line(1,0){40}}
\put(155,230){\line(0,1){50}}
\put(115,230){\line(2,1){20}}
\put(115,280){\line(2,-1){20}}
\put(155,280){\line(-2,-1){20}}
\put(155,230){\line(-2,1){20}}

\put(170,251){$=$}
\put(123,240){$A$}
\put(138,260){$B$}

\put(215,270){\circle*{4}}
\put(215,240){\circle*{4}}
\put(195,280){\circle*{4}}
\put(235,280){\circle*{4}}
\put(195,230){\circle*{4}}
\put(235,230){\circle*{4}}
\put(195,230){\line(0,1){50}}
\put(195,230){\line(1,0){40}}
\put(195,280){\line(1,0){40}}
\put(235,230){\line(0,1){50}}
\put(215,240){\line(0,1){30}}
\put(195,230){\line(2,1){20}}
\put(195,280){\line(2,-1){20}}
\put(235,280){\line(-2,-1){20}}
\put(235,230){\line(-2,1){20}}
\put(195,280){\line(1,-2){20}}
\put(235,280){\line(-1,-2){20}}

\put(250,251){$-$}
\put(200,240){$A$}
\put(215,260){$B$}

\put(295,270){\circle*{4}}
\put(295,240){\circle*{4}}
\put(275,280){\circle*{4}}
\put(315,280){\circle*{4}}
\put(275,230){\circle*{4}}
\put(315,230){\circle*{4}}
\put(275,230){\line(0,1){50}}
\put(275,230){\line(1,0){40}}
\put(275,280){\line(1,0){40}}
\put(315,230){\line(0,1){50}}
\put(275,230){\line(2,1){20}}
\put(275,280){\line(2,-1){20}}
\put(315,280){\line(-2,-1){20}}
\put(315,230){\line(-2,1){20}}
\put(275,280){\line(1,-2){20}}
\put(315,280){\line(-1,-2){20}}

\put(280,240){$A$}
\put(295,260){$B$}

\put(110,215){${\scriptstyle u^6 +8u^4 +12u^2}$}
\put(190,215){${\scriptstyle u^6 +11u^4 +8u^2}$}
\put(270,215){${\scriptstyle u^6 +10u^4 +8u^2}$}

\put(55,170){\circle*{4}}
\put(55,140){\circle*{4}}
\put(35,180){\circle*{4}}
\put(75,180){\circle*{4}}
\put(35,130){\circle*{4}}
\put(75,130){\circle*{4}}
\put(35,130){\line(0,1){50}}
\put(35,130){\line(1,0){40}}
\put(35,180){\line(1,0){40}}
\put(75,130){\line(0,1){50}}
\put(55,140){\line(0,1){30}}
\put(35,130){\line(2,1){20}}
\put(35,180){\line(2,-1){20}}
\put(75,180){\line(-2,-1){20}}
\put(75,130){\line(-2,1){20}}

\put(90,151){$-$}
\put(43,140){$A$}
\put(78,125){$B$}

\put(135,170){\circle*{4}}
\put(135,140){\circle*{4}}
\put(115,180){\circle*{4}}
\put(155,180){\circle*{4}}
\put(115,130){\circle*{4}}
\put(155,130){\circle*{4}}
\put(115,130){\line(0,1){50}}
\put(115,130){\line(1,0){40}}
\put(115,180){\line(1,0){40}}
\put(155,130){\line(0,1){50}}
\put(135,140){\line(0,1){30}}
\put(115,130){\line(2,1){20}}
\put(115,180){\line(2,-1){20}}
\put(155,180){\line(-2,-1){20}}

\put(170,151){$=$}
\put(123,140){$A$}
\put(158,125){$B$}

\put(215,170){\circle*{4}}
\put(215,140){\circle*{4}}
\put(195,180){\circle*{4}}
\put(235,180){\circle*{4}}
\put(195,130){\circle*{4}}
\put(235,130){\circle*{4}}
\put(195,130){\line(0,1){50}}
\put(195,130){\line(1,0){40}}
\put(195,180){\line(1,0){40}}
\put(235,130){\line(0,1){50}}
\put(215,140){\line(0,1){30}}
\put(195,180){\line(2,-1){20}}
\put(235,180){\line(-2,-1){20}}
\put(235,130){\line(-2,1){20}}
\put(235,180){\line(-1,-2){20}}
\put(235,180){\line(-1,-2){20}}

\put(250,151){$-$}
\put(203,140){$A$}
\put(238,125){$B$}

\put(295,170){\circle*{4}}
\put(295,140){\circle*{4}}
\put(275,180){\circle*{4}}
\put(315,180){\circle*{4}}
\put(275,130){\circle*{4}}
\put(315,130){\circle*{4}}
\put(275,130){\line(0,1){50}}
\put(275,130){\line(1,0){40}}
\put(275,180){\line(1,0){40}}
\put(315,130){\line(0,1){50}}
\put(295,140){\line(0,1){30}}
\put(275,180){\line(2,-1){20}}
\put(315,180){\line(-2,-1){20}}
\put(315,180){\line(-1,-2){20}}

\put(283,140){$A$}
\put(318,125){$B$}

\put(105,115){${\scriptstyle u^6 +8u^4 +10u^2 +1}$}
\put(190,115){${\scriptstyle u^6 +9u^4 +11u^2}$}
\put(265,115){${\scriptstyle u^6 +8u^4 +9u^2 +1}$}

\end{picture}
\end{center}

		\caption{$4$-term relations for the $5$-wheel expressing them in terms
of intersection graphs; the value of the skew characteristic polynomial is indicated}
		\label{fR5}
\end{figure}

\begin{theorem}
The refined skew characteristic polynomial of simple graphs satisfies the $4$-term relations for graphs.
\end{theorem}

The proof of the theorem is similar to that for skew characteristic polynomial and is
based on the fact that the corank of the adjacency matrix of a~graph is preserved by the second
Vassiliev move, see equation~\eqref{eSVmm} below, that is, satisfies the $2$-term relations.

\subsection{Characteristic polynomial of antisymmetric matrices
and \texorpdfstring{$4$}{4}-term relations}\label{ssasm}

Symmetric matrices over the field of two elements, considered up to simultaneous permutations of rows and columns, are in one-to-one correspondence with isomorphism classes of simple graphs. Following the pattern for adjacency matrices of graphs in the 4-term relation,
we are able to define the 4-term relation for arbitrary square matrices. It happens that, for antisymmetric matrices,
the characteristic polynomial satisfies these $4$-term relations. We hope that this property will find applications outside
the theory of finite type knot invariants.

Let $M_n$ be the set of $n \times n$ matrices over the base field, considered up to simultaneous permutations of rows and columns.
Let $f$ be a~function on $M_n$. We say that \textit{$f$ satisfies the 4-term relations} if for any matrix $A=(a_{ij})\in M_n$ we have

\[
f(A)-f(A_{12}')=f(\widetilde{A}_{12})-f(\widetilde{A}'_{12}),
\]
where
\begin{itemize}
    \item the matrix $A'_{12}$ is the matrix obtained from $A$ by replacing the two entries
$a_{12}$ and $a_{21}$ with zeroes;
\item the matrix $\widetilde{A}_{12}$ is
\begin{equation}\label{eSVmm}
\widetilde{A}_{12}=B^t AB,
\end{equation}
where $B$ is the block square matrix consisting of a~$2\times2$-block
\[
\left(
\begin{array}
{cc}
1&1\\0&1
\end{array}
\right)
\]
in the first two rows and columns and a~complementary identity block;
\item the matrix $\widetilde{A}'_{12}$ is the result of applying to~$A$ the operations, tilde and prime, in the elements~$1$ and~$2$, and is independent of the order of the two.
\end{itemize}

 We prove that the characteristic polynomial satisfies this 4-term relation for antisymmetric matrices of order $n$.
 In particular, this is true for adjacency matrices of oriented simple graphs. It is also true
 for adjacency matrices of non-oriented simple graphs, if we consider them over the field of
 two elements, since in this case symmetric matrices are also antisymmetric.
 It is easy to check that for symmetric matrices, for example, the corresponding assertion is wrong.

\begin{theorem}
    The characteristic polynomial satisfies the $4$-term relations for the space of antisymmetric $n\times n$-matrices.
\end{theorem}

\begin{proof}
Let $A$ be an antisymmetric matrix of order $n$, and
let $\chi_A$ be its characteristic polynomial, i.e. $\chi_A(u)=\lvert u I-A \rvert$.
Suppose {\small
        \setlength{\arraycolsep}{0.05cm}
\[
A=
    \begin{pmatrix}
        0 & -a_{12} & -a_{13} & \ldots & -a_{1n}\\
        a_{12} & 0 & -a_{23} & \ldots & -a_{2n}\\
        \vdots & \vdots & \ddots & & \vdots\\
        a_{1n} & a_{2n} & \ldots & \ldots & 0
    \end{pmatrix}
\]
}

    Then, {\small
        \setlength{\arraycolsep}{0.05cm}
\[
        A'=
    \begin{pmatrix}
        0 & 0 & -a_{13} & \ldots & -a_{1n}\\
        0 & 0 & -a_{23} & \ldots & -a_{2n}\\
        \vdots & \vdots & \ddots & & \vdots\\
        a_{1n} & a_{2n} & \ldots & \ldots & 0
    \end{pmatrix}
,
    \quad \Tilde{A}=
        \begin{pmatrix}
        0 & -a_{12} & -a_{13}-a_{23} & \ldots & -a_{1n}-a_{2n}\\
        a_{12} & 0 & -a_{23} & \ldots & -a_{2n}\\
        a_{13}+a_{23} & a_{23} & 0 & \ldots & -a_{2n}\\
        \vdots & \vdots & \ddots & & \vdots\\
        a_{1n}+a_{2n} & a_{2n} & \ldots & \ldots & 0
    \end{pmatrix}
, \text{ and}
   \]

   \[
   \Tilde{A}'=
        \begin{pmatrix}
        0 & 0 & -a_{13}-a_{23} & \ldots & -a_{1n}-a_{2n}\\
        0 & 0 & -a_{23} & \ldots & -a_{2n}\\
        a_{13}+a_{23} & a_{23} & 0 & \ldots & -a_{2n}\\
        \vdots & \vdots & \ddots & & \vdots\\
        a_{1n}+a_{2n} & a_{2n} & \ldots & \ldots & 0
    \end{pmatrix}
.
   \]
}

   We have
   {\small
        \setlength{\arraycolsep}{0.05cm}
\[
        \chi_A-\chi_{A'}=
        \begin{vmatrix}
        u & a_{12} & a_{13} & \ldots & a_{1n}\\
        -a_{12} & u & a_{23} & \ldots & a_{2n}\\
        \vdots & \vdots & \ddots & & \vdots\\
        -a_{1n} & -a_{2n} & \ldots & \ldots & u
    \end{vmatrix}
    -
    \begin{vmatrix}
        u & 0 & a_{13} & \ldots & a_{1n}\\
        0 & u & a_{23} & \ldots & a_{2n}\\
        \vdots & \vdots & \ddots & & \vdots\\
        -a_{1n} & -a_{2n} & \ldots & \ldots & u
    \end{vmatrix}
    \]
}

    Expanding $\chi_A$ along the first row, and $\chi_{A'}$ along the first column, we get
           {\footnotesize
        \setlength{\arraycolsep}{0.03cm}
        \allowdisplaybreaks
        \begin{gather*}
        \chi_A-\chi_{A'}=
                 -a_{12}
\begin{vmatrix}
        -a_{12} & a_{23} & a_{24} & \ldots & a_{2n}\\
        -a_{13} & u & a_{34} & \ldots & a_{3n}\\
        \vdots & -a_{34} & u & \ldots & \vdots\\
        \vdots & \vdots & &\ddots &\vdots\\
        -a_{1n} & -a_{3n} & -a_{4n} & \ldots & u
        \end{vmatrix}
+a_{13}\left(
\begin{vmatrix}
        -a_{12} & u & a_{24} & \ldots & a_{2n}\\
        -a_{13} & -a_{23} & a_{34} & \ldots & a_{3n}\\
        \vdots & \vdots & u & \ldots & \vdots\\
        \vdots & \vdots & &\ddots &\vdots\\
        -a_{1n} & -a_{2n} & -a_{4n} & \ldots & u
        \end{vmatrix}
+
        \begin{vmatrix}
        0 & a_{13} & \ldots & \ldots & a_{1n}\\
        u & a_{23} & \ldots & \ldots & a_{2n}\\
        -a_{24} & -a_{34} & & \ldots & a_{4n}\\
        \vdots & \vdots & &\ddots &\vdots\\
        -a_{2n} & -a_{3n} & -a_{4n} & \ldots & u
        \end{vmatrix}
\right) \\ +\dots + (-1)^{k-1}a_{1k}\left(
\begin{vmatrix}
        -a_{12} & u & a_{23} & \ldots & a_{2(k-1)} & a_{2(k+1)} & \ldots & a_{2n}\\
        -a_{13} & -a_{23} & u & & & & \ldots & a_{3n}\\
        \vdots & \vdots & \ldots& & & & \ldots & \vdots\\
        \vdots & \vdots & \ldots & & & &\ddots &\vdots\\
        -a_{1n} & -a_{2n} & \ldots & & -a_{(k-1)n} & -a_{(k+1)n} & \ldots &u
        \end{vmatrix}
+
\begin{vmatrix}
        0 & a_{13} & \ldots & \ldots & a_{1n}\\
        u & a_{23} & \ldots & \ldots & a_{2n}\\
        -a_{23} & u & \ldots & & a_{3n}\\
        -a_{2(k-1)} & \vdots & & \ldots & a_{(k-1)n}\\
        -a_{2(k+1)} & \vdots & \ddots & & a_{(k+1)n}\\
        \vdots & \vdots & &\ddots &\vdots\\
        -a_{2n} & -a_{3n} & -a_{4n} & \ldots & u
        \end{vmatrix}
\right) \\ +\ldots + (-1)^{n-1}a_n\left(
       \begin{vmatrix}
        -a_{12} & u & a_{23} & \ldots & a_{2(n-1)}\\
        -a_{13} & -a_{23} & u & \ldots & a_{3(n-1)}\\
        \vdots & \vdots & u & \ldots & \vdots\\
        \vdots & \vdots & &\ddots &u\\
       -a_{1n} & -a_{2n} & -a_{3n} & \ldots & -a_{(n-1)n}
        \end{vmatrix}
\right.
        + \left.
\begin{vmatrix}
        0 & a_{13} & \ldots & \ldots & a_{1n}\\
        u & a_{23} & \ldots & \ldots & a_{2n}\\
        -a_{23} & u & & \ldots & a_{3n}\\
        \vdots & \vdots & &\ddots &\vdots\\
        -a_{2(n-1)} & -a_{3(n-1)} & \ldots & u & a_{(n-1)n}
        \end{vmatrix}
\right).
        \end{gather*}
}
        Now, we make some observations:
    \begin{enumerate}
        \item For $k\geq 3$, in each summand, there is a~difference term of two $(n-1)\times (n-1)$ minors of $a_{1k}$, coming from $\chi_A$ and $\chi_{A'}$ respectively. We expand one with respect to the first row, and the other with respect to the first column.
        \item It follows that, in each summand, the terms corresponding to the expansion wrt $a_{1k}$ cancel out.
        \item Also, for $k,l \geq 3$, $k\neq l$,
\[
(-1)^{n-1}(a_{1k})(a_{1l})\Big(|A_1|-|A_2| \Big) = -(-1)^{n-1}(a_{1l})(a_{1k})\Big(|A_2|-|A_1| \Big),
\]
        where $A_1$, $A_2$ are the minors of $a_{1l}$, in the expansion inside the summand corresponding to $a_{1k}$. Note that they are also the minors of $a_{1k}$, in the expansion inside the summand corresponding to $a_{1l}$.
        Therefore, they also cancel out.

    \end{enumerate}
    Hence, we are left with the following expression,
    \begin{align} \label{eq:1}
         \chi_A-\chi_{A'} 
         ={ }&{ } -a_{12}
        \begin{vmatrix}
            -a_{12} & a_{23} & a_{24} & \ldots & a_{2n}\\
            -a_{13} & u & a_{34} & \ldots & a_{3n}\\
            \vdots & -a_{34} & u & \ldots & \vdots\\
            \vdots & \vdots & &\ddots &\vdots\\
            -a_{1n} & -a_{3n} & -a_{4n} & \ldots & u
        \end{vmatrix} + (a_{13})(-a_{12})
        \begin{vmatrix}
            -a_{23} & a_{34} & \ldots & a_{3n}\\
            \vdots & u & \ldots & \vdots\\
            \vdots & & \ddots & \vdots\\
            -a_{2n} & \ldots & \ldots & u
        \end{vmatrix} \\
        &{\quad}+ \dotsb
        + (-1)^{k-1}(a_{1k})(-a_{12})
        \begin{vmatrix}
            -a_{23} & u & \ldots & a_{3n}\\
            \vdots & & \ddots & \vdots\\
            -a_{2n} & & \ldots & u
        \end{vmatrix} \nonumber\\
        &{\quad}+ \dotsb
        +(-1)^{n-1}(a_{1n})(-a_{12})
        \begin{vmatrix}
            -a_{23} & u & \ldots & a_{3(n-1)}\\
            \vdots & & \ddots & \vdots\\
            \vdots & & \ldots & u\\
            -a_{2n} & -a_{3n} & \ldots & -a_{(n-1)n}
        \end{vmatrix}, \nonumber
    \end{align}
where the minor in the first summand, is of order $(n-1)\times (n-1)$, and in the rest of the summands are of order $(n-2)\times (n-2)$.

        Now, consider an analogous procedure for $\chi_{\Tilde{A}}-\chi_{\Tilde{A}'}$:
           {\small
        \setlength{\arraycolsep}{0.05cm}
\begin{align*}
\chi_{\Tilde{A}}-\chi_{\Tilde{A}'} & =
        -a_{12}
\begin{vmatrix}
        -a_{12} & a_{23} & a_{24} & \ldots & a_{2n}\\
        -(a_{13}+a_{23}) & u & a_{34} & \ldots & a_{3n}\\
        \vdots & -a_{34} & u & \ldots & \vdots\\
        \vdots & \vdots & &\ddots &\vdots\\
        -(a_{1n}+a_{2n}) & -a_{3n} & -a_{4n} & \ldots & u
        \end{vmatrix}
        + (a_{13}+a_{23})(-a_{12})
\begin{vmatrix}
        -a_{23} & a_{34} & \ldots & a_{3n}\\
        \vdots & u & \ldots & \vdots\\
        \vdots & & \ddots & \vdots\\
        -a_{2n} & \ldots & \ldots & u
        \end{vmatrix}
+ \dots\\
        & \qquad \qquad\qquad \qquad\dots + (-1)^{k-1}(a_{1k}+a_{2k})(-a_{12})
\begin{vmatrix}
        -a_{23} & u & \ldots & a_{3n}\\
        \vdots & & \ddots & \vdots\\
                -a_{2n} & & \ldots & u
        \end{vmatrix}
+ \dots\\
        & \qquad \qquad\qquad\qquad \dots +(-1)^{n-1}(a_{1n}+a_{2n})(-a_{12})
\begin{vmatrix}
        -a_{23} & u & \ldots & a_{3(n-1)}\\
        \vdots & & \ddots & \vdots\\
        \vdots & & \ldots & u\\
        -a_{2n} & -a_{3n} & \ldots & -a_{(n-1)n}
        \end{vmatrix}
.
        \end{align*}
}
        Now, in the above expression of $\chi_{\Tilde{A}}(u)-\chi_{\Tilde{A}'}(u)$, expand the determinant in the first summand along the first column. Then a~straightforward calculation shows that the summands can be rearranged to form an expression that is exactly equal to the right-hand side of equation~(\ref{eq:1}).
\end{proof}

As an immediate consequence of this theorem, we obtain

\begin{corollary}
   The skew characteristic polynomial $Q_C$ of chord diagrams satisfies the $4$-term relations and is hence a~weight system.
\end{corollary}

\section{Skew characteristic polynomial for delta-matroids and embedded graphs}\label{s3}

In this section, we extend the skew characteristic polynomial to
a two-variable polynomial invariant of embedded graphs
and delta-matroids. Chord diagrams can be considered as embedded graphs with a~single
vertex (the vertex corresponds to the underlying circle of the chord diagram,
while the chords are the edges). In contrast to chord diagrams,
embedded graphs with more than one vertex do not span a~Hopf algebra in a~natural way,
even when considered modulo $4$-term relations.
Fortunately, however, a~construction due to A.~Bouchet allows one to associate
to each graph and each embedded graph its delta-matroid. In turn,
delta-matroids span a~Hopf algebra~\cite{Lando2016DeltamatroidsAV}, and one can define the skew characteristic polynomial
for delta-matroids following the same pattern as above.
The construction makes use of the fact that we know how to extend non-degeneracy to delta-matroids.

Our approach to extending the skew characteristic polynomial
to embedded graphs and delta-matroids is similar in nature to that of~\cite{NZ},
where an extension of Stanley's symmetrized chromatic polynomial has been constructed.
M.~Nenasheva and V.~Zhukov are basing more specifically on combinatorial Hopf algebra structures,
and constructed a~character for the extended invariant, which is a~graded Hopf algebra
homomorphism. For the skew characteristic polynomial, this is not true, and we use
only the knowledge of the value of the extended invariant on primitives.
Extensions of graph invariants to delta-matroids in their relationship
with knot and link invariants are studied also in~\cite{Dunaykin2019TransitionPA,K20}.

The $4$-term relations also admit an extension to embedded graphs and binary delta-matroids.
Since non-degeneracy respects these relations, the extended skew characteristic polynomial
also satisfies them, whence determining an invariant of multi-component links.

Below, we recall necessary knowledge about embedded graphs and delta-matroids.
We follow mainly~\cite{Chun2018OnTI}.

\subsection{Embedded graphs}

\begin{definition}
    A \emph{ribbon graph} (or a~\emph{cellularly embedded graph}) $\Gamma=(V,E)$, is a~surface with boundary
     such that it can be represented as the union of two sets of disks, namely,
     a~set $V$ of vertices, and a~set $E$ of edges such that:
\begin{itemize}
    \item The vertices and edges intersect in disjoint line segments;
    \item Each such line segment lies on the boundary of exactly one vertex, and exactly one edge;
   \item Every edge contains exactly two such line segments.
\end{itemize}
\end{definition}
This surface can be orientable or non-orientable.
Below, we restrict ourselves to the case of oriented embedded graphs.

We can form ribbon graphs (with one vertex) from chord diagrams by attaching a~disk to the outer circle and thickening the chords to form ribbons and hence inducing a~natural cyclic order to the ribbons.
 If we consider arbitrary embedded graphs, allowing for more than one vertex,
 we are led to obtain a~generalization of the concept of weight systems for knots, expanding it to a~notion of weight systems for links.
 The number of components of a~link then correspond to the number of vertices in the embedded graph.

For an arbitrary embedded graph $\Gamma=(V,E)$, it is not conducive to assign an intersection graph to it, so we focus on the associated so-called `delta-matroid' of the embedded graph $\Gamma$. This delta-matroid is binary.
First, some preliminaries on delta-matroids are discussed. A detailed study can be found in~\cite{Chun2018OnTI}.

\subsection{Delta-matroids of graphs and embedded graphs}

\begin{definition}[Set System]
    A \emph{set system} is a~pair $D=(E;\Phi)$, where $E$, called the \emph{ground set}, is a~finite set,
    and $\Phi \subset 2^E$, called the set of \emph{feasible sets}, is a~subset of the family of subsets of $E$.
\end{definition}

A set system $D=(E;\Phi)$ is `proper' if $\Phi$ is nonempty.

\begin{definition}[Delta-Matroid]
    A proper set system $D=(E;\Phi)$ is called a~\emph{delta-matroid} if it satisfies the following `Symmetric Exchange Axiom' (SEA):
    \[
\forall \text{ X,Y} \in \Phi, \forall a\in X\triangle Y, \exists b\in X\triangle Y \text{ such that } X\triangle \{a,b\}\in \Phi.
\]
\end{definition}

\begin{example}
      Let $E=\{a,b,c\}$, and $\Phi=\{ \{a\},\{b\},\{a,b\},\{a,b,c\}\}$; then $D=(E;\Phi)$ is a~delta-matroid.
\end{example}
A delta-matroid $(E;\Phi)$ is called \textit{even} if for any pair of feasible sets $X,Y\in\Phi$, we have $\lvert X\rvert \equiv \lvert Y\rvert \mod 2$.

To a~given simple graph $G=(V,E)$ with vertex set $V$ and edge set $E$, we can associate its \emph{graphic delta-matroid}
$D(G)=(V(G);\Phi(G))$, also called the \textit{non-degeneracy delta-matroid} of~$G$
with ground set $E(D(G)) = V(G)$ and admissible sets $\Phi(G)$ consisting of all $A\subset V(G)$
such that the adjacency matrix of the subgraph $G(A)$ induced by $A$ is non-degenerate.

\begin{definition}[Local duality/Partial dual/twist]
    For a~delta-matroid $D=(E;\Phi)$, and a~subset $A\subset E$, the partial dual of $D$ with respect to $A$ is defined as
    \[
D\ast A = (E;\Phi\ast A), \text{ where } \Phi\ast A = \{F\triangle A : F\in \Phi\}.
\]
\end{definition}
The above operation $\ast$ is also called the \textit{twist} operation. For embedded graphs,
it was introduced in~\cite{Chmutov2009}.

\begin{definition}[Binary Delta-Matroid]
    A delta-matroid $D=(E;\Phi)$ is \emph{binary} if there exists $A\subset E$ such that $D\ast A$ is a~graphic delta-matroid.
\end{definition}

\begin{definition}
    A \emph{quasi-tree} is a~ribbon graph with exactly one boundary component. A ribbon subgraph of a~connected ribbon graph $\Gamma$ is a~\emph{spanning quasi-tree} of $\Gamma$ if it is a~quasi-tree and has the same vertex set as $\Gamma$.
\end{definition}

\begin{definition}[Ribbon graphic delta-matroid]
    Let $\Gamma=(V,E)$ be a~ribbon graph and let
\[
\Phi(\Gamma):=\{F\subset E(\Gamma) : F \text{ is the edge set of a~spanning quasi-tree of }\Gamma\};
\]
then $D(\Gamma):=(E(\Gamma);\Phi(\Gamma))$ is the \emph{delta-matroid of the ribbon graph} $\Gamma$.
\end{definition}

We state some well-known properties of ribbon graphs and their corresponding delta-matroids in the form of a~lemma, without proof.
\begin{lemma}
{(see~\cite{Bouchet1989MapsAD})}
\label{7.1}
\begin{enumerate}
    \item The delta-matroid $D(\Gamma)$ of a~ribbon graph $\Gamma$ is binary.
    \item The delta-matroid $D(C)$, where $C$ is a~ribbon graph with a~single vertex, is isomorphic to the graphic delta-matroid associated to the intersection graph $g(C)$ of the chord diagram $C$.
    \item $D(\Gamma)$ is an even binary delta-matroid if and only if the ribbon graph $\Gamma$ is orientable.
    \item Given a~delta-matroid $D(\Gamma)=(E=E(\Gamma);\Phi)$ and $A\subset E$, $D\ast A$ is the delta-matroid of the partial dual \textrm{\cite{Chun2018OnTI}} of the ribbon graph $\Gamma$ with respect to $A$, $\Gamma^A$.
\end{enumerate}
\end{lemma}

\subsection{The Hopf algebra of delta-matroids}

Two set systems are said to be isomorphic if there is a~one-to-one correspondence between
their ground sets taking the feasible sets of the first set system to that of the second
one and vice versa.
Isomorphism classes of set systems span an infinite dimensional graded vector space:
\[
\mathcal{S} = \mathcal{S}_0 \oplus \mathcal{S}_1 \oplus \mathcal{S}_2 \oplus \hdots,
\]
where $\mathcal{S}_i$ is the vector space spanned by isomorphism classes of set systems
with ground set of cardinality $i$.
Given two set systems $D_1=(E_1;\Phi_1)$, $D_2=(E_2;\Phi_2)$, their \textit{product} is defined as
\[
D_1D_2=(E_1 \sqcup E_2;\{X \sqcup Y : X\in \Phi_1, Y\in \Phi_2\}).
\]
This product extends to $\mathcal{S}$ by linearity to a~graded bilinear multiplication
\[
m:\mathcal{S}\otimes \mathcal{S} \rightarrow \mathcal{S}, \hspace{1.0cm} m(D_1 \otimes D_2) = D_1D_2.
\]
The unit with respect to $m$ is given by $(\emptyset;\{\emptyset\})$, and it generates $\mathcal{S}_0$.

The product of two binary delta-matroids is a~binary delta-matroid~\cite{Chun2018OnTI}.
Thus, the vector space of (isomorphism classes of)
binary delta-matroids, $\B$, forms a~graded commutative subalgebra of the algebra of set systems $\mathcal{S}$:
\[
\B = \B_0 \oplus \B_1 \oplus \B_2 \hdots
\]

Similarly, the vector space of even binary delta-matroids, $\B^e$, forms a~graded subalgebra of $\B$:
\[
\B^e = \B_0^e \oplus \B_1^e \oplus \B_2^e \oplus \hdots
\]

If an embedded graph~$\Gamma$ is obtained as a~result of gluing
    a~vertex chosen arbitrarily in an embedded graph~$\Gamma_1$ to a~vertex chosen arbitrarily in
    an embedded graph $\Gamma_2$, then the delta-matroid of~$\Gamma$ is the product of delta-matroids of~$\Gamma_1$
    and~$\Gamma_2$: $D(\Gamma)=D(\Gamma_1)D(\Gamma_2)$.

Given a~delta-matroid $D=(E;\Phi)$, its \textit{coproduct} $\mu(D)$ is given as
 \[
\mu(D) = \sum_{E'\subset E} D(E') \otimes D(E\setminus E'),
\]
where $D(E')$ is the \textit{restriction} of a~delta-matroid to the subset $E'$ of $E$. It is easily proved that $\mu(D_1D_2) = \mu(D_1)\mu(D_2)$. The notion of a~restriction for delta-matroids is described below.
The coproduct is extended by linearity to a~comultiplication

\[
\mu : \B \rightarrow \B\otimes \B,
\]
which respects the property of being even.

Let $D=(E;\Phi)$ be a~delta matroid. A \textit{coloop} is defined to be an element $e\in E$ such that $\forall X\in \Phi, e\in X$; and a~\textit{loop} is an element $e\in E$ such that $\forall X\in \Phi, e\notin X$.

\begin{definition}[Deletion]
    Given a~delta-matroid $D=(E;\Phi)$ and an element $e\in E$, the \emph{deletion} $D\setminus e$ is the set system defined as $D\setminus e:= (E\setminus e; \Phi')$, where,
    if $e$ is not a~coloop, then
    \[
\Phi' = \{X: X\in \Phi: e\notin X\},
\]
    and if $e$ is a~coloop, then
    \[
\Phi' = \{X\setminus e: X\in \Phi, e\in X\}.
\]
\end{definition}

\begin{definition}[Contraction]
    Given a~delta-matroid $D=(E;\Phi)$ and an element $e\in E$, the \emph{contraction} $D/ e$ is the set system defined as $D/ e:= (E\setminus e; \Phi')$, where,
    if $e$ is not a~loop, then
    \[
\Phi' = \{X\setminus e: X\in \Phi: e\in X\},
\]
    and if $e$ is a~loop, then
    \[
\Phi' = \Phi.
\]
\end{definition}

\begin{lemma}
    If $D$ is a~delta-matroid, then $D\setminus e$ and $D/e$ are also delta-matroids, and $D/e = (D\ast e)\setminus e$.
\end{lemma}

Hence, we can define $D\setminus A$ and $D/ A$, for a~given subset $A$ of the ground set as
the result of a~sequence of contractions or deletions with respect to all elements in $A$.
In particular, the \emph{restriction} of a~delta-matroid to a~subset $E'$ of its ground set $E$, is the delta-matroid obtained
by deleting all elements in $E\setminus E'$.

\begin{proposition}
    If $D(\Gamma) = (E(\Gamma); \Phi(\Gamma))$ is the delta-matroid of an embedded graph $\Gamma$ and $E' \subset E(\Gamma)$ is a~subset of its edges such that the corresponding spanning subgraph is connected, then $D(E')$ is the delta-matroid of the spanning subgraph $(V(\Gamma), E')$.
    \end{proposition}

\begin{proposition}
    For a~binary delta-matroid $D=(E; \Phi)$, its restriction $D(E')$ to an arbitrary subset $E'\subset E$ is a~binary delta-matroid.
\end{proposition}

\begin{theorem}
    The vector space $\B$ endowed with the multiplication and comultiplication defined above forms a~graded commutative cocommutative Hopf Algebra. The subalgebra $\B^e \subset \B$ spanned by even binary delta-matroids forms a~Hopf subalgebra in this Hopf algebra.
\end{theorem}

\subsection{Skew characteristic polynomial for delta-matroids}

Motivated by the correspondence between embedded graphs and binary delta-matroids,
we can generalize the skew characteristic polynomial to the bialgebra~$\B$ of binary delta-matroids.

Note that the non-degeneracy $\nu$ admits a~natural extension to binary delta-matroids:
if $D=(E;\Phi)$ is a~binary delta-matroid, define $\nu(D) = 1$ if $E \in \Phi$ and $0$ otherwise.
Indeed, for a~chord diagram~$C$, its intersection graph~$g(C)$ is non-degenerate if and only if
the boundary of~$C$ is connected, that is, if and only if~$C$ is a~quasi-tree.

\begin{definition}
The \emph{skew characteristic polynomial} $Q:D\mapsto Q_D(u,w)$ is the invariant of even binary delta-matroids
taking values in the
ring $\C[u,w]$ of polynomials in two variables~$u,w$ possessing the following properties:
\begin{itemize}
\item it is multiplicative, that is if $D$ is a~product of two even binary delta-matroids $D_1$ and $D_2$,
$D=D_1D_2$, then $Q_{D}=Q_{D_1}\cdot Q_{D_2}$;
\item for the delta-matroid $(\{1\};\{\emptyset\})$ corresponding to an orientable embedded
graph with a~single vertex and a~single edge, the polynomial is~$u$;
\item for the delta-matroid $(\{1\};\{\{1\}\})$ corresponding to the embedded
graph with two vertices and a~single edge (such an embedded graph is necessarily orientable), the polynomial is~$w+1$;
\item for any even binary delta-matroid~$D$ whose ground set contains at least two elements,
the value $Q_{\pi(D)}$ of the invariant~$Q$ on the projection $\pi(D)$ to the subspace of primitives
along the subspace of decomposable elements in the
Hopf algebra~$\B^e$ of even binary delta-matroids is a~constant;
\item the free term of the polynomial~$Q_D$ coincides with $\nu(D)$.
\end{itemize}
\end{definition}

\begin{remark} We chose the value of the skew characteristic polynomial on the delta-matroid $(\{1\};\{\{1\}\})$
to be $w+1$ rather than just~$w$ in order to make it consistent with the requirement
that the free term of the polynomial coincides with the non-degeneracy of the delta-matroid.
\end{remark}

Since the set~$V(G)$ of vertices of a~graph~$G$ is admissible for the delta-matroid of
the graph if and only if $\nu(G)=1$, we deduce

\begin{theorem}
  The skew characteristic polynomial of a~graph~$G$ coincides with the skew characteristic polynomial
  of its delta-matroid.
\end{theorem}

The notions of the first and the second Vassiliev moves for binary delta-matroids,
as well as $4$-term relations for invariants of binary delta-matroids,
where introduced in~\cite{Lando2016DeltamatroidsAV}.
Similarly to the case of chord diagrams and graphs, since the non-degeneracy
of delta-matroids satisfies $2$-term relations, we immediately deduce

\begin{theorem}
The skew characteristic polynomial of even binary delta-matroids satisfies
$4$-term relations and defines, therefore, a~finite type link invariant.
\end{theorem}

\subsection{Examples}

The delta-matroid $(\{1,2\};\{\{1\},\{2\}\})$ corresponds to the orientable embedded
graph with two vertices connected by two ribbons. Its projection to primitives looks
like
\[
\pi:(\{1,2\};\{\{1\},\{2\}\})\mapsto (\{1,2\};\{\{1\},\{2\}\})-(\{1\};\{\{1\}\})^2.
\]
This gives the value
\[
Q_{(\{1,2\};\{\{1\},\{2\}\})}=(w+1)^2 -1=w^2 +2w.
\]

More generally, for the delta-matroid of the orientable embedded graph of genus~$0$
consisting of two vertices connected by~$n$ ribbons, the skew characteristic polynomial is
\[
w^n +nw^{n-1}.
\]

\section{Problems and questions}\label{s4}

The following questions arise naturally in the study of any newly introduced
invariant of graphs and delta-matroids:
\begin{itemize}
  \item Compute explicitly the skew characteristic polynomials for important series of
  graphs and delta-matroids;
  \item Isolate properties of graphs and delta-matroids reflected in their skew characteristic
  polynomials;
  \item Try to find out natural contraction-deletion relations allowing one
  to compute the skew characteristic polynomial more efficiently than
  directly from the definition;
  \item Construct an analogue of the Laplacian operator whose spectral
  theory is governed by the skew characteristic polynomial;
  \item Categorify the skew characteristic polynomial.
\end{itemize}

\EditInfo{May 14, 2023}{September 20, 2023}{Jacob Mostovoy and Sergei Chmutov}

\end{document}